\documentclass[a4paper,12pt,reqno]{amsart}
\usepackage{amsmath,amssymb,amsthm}
\usepackage{mathrsfs}
\usepackage{enumerate}
\usepackage{graphicx}
\numberwithin{equation}{section}
\theoremstyle{plain}
\newtheorem{thm}{Theorem}[section]
\theoremstyle{definition}

\newtheorem{exam}[thm]{Example}
\newtheorem{Lemma}[thm]{Lemma}
\newtheorem{pro}{Proof}
\usepackage{mathrsfs}

\begin{document}
\title{Moment Asymptotic Expansions of the Wavelet Transforms}
\author[R.S.Pathak and Ashish Pathak]{ R S Pathak* and Ashish Pathak** \\
 *DST Center for Interdisciplinary Mathematica Sciences \\
  Banaras Hindu University , Varanasi, india-221005\\
**Department of Mathematics and Statistics\\
 Dr. Harisingh Gour Central University \\
    Sagar-470003, India.}
\date{}
\thanks{$*$E-mail: ramshankarpathak@yahoo.co.in $^{**}$E-mail: pathak\_maths@yahoo.com}
\begin{abstract} Using distribution theory we present the moment asymptotic expansion of continuous wavelet transform in different distributional spaces for large and small values of dilation parameter $a$. We also obtain asymptotic expansions for certain wavelet transform.
\end{abstract}
\maketitle
\section{Introduction} 
In past few decades their were many mathematician who has done great work in the field of asymptotic expansion like Wong 1979 \cite{wong} using Mellin transform technique has obtained asymptotic expansion of classical integral transform and after that Pathak $\&$ Pathak 2009 \cite{rspathak,Pathak,APathak,Aapathak} has found the asymptotic expansion of continuous wavelet transform for large and small values of dilation and translation parameters. Estrada $\&$ Kanwal 1990 \cite{kanwal} has obtained the asymptotic expansion of generalized functions on different  spaces of test functions. In present paper using  Estrada $\&$ Kanwal technique we have obtained the asymptotic expansion of wavelet transform in different distributional spaces.\\
\hspace*{2 mm} The continuous wavelet transform of $f$ with respect to wavelet $\psi$ is defined by
\begin{eqnarray}
\label{1}
\left(  W_\psi f \right)  (a,b)= \frac{1}{\sqrt{a}} \int_{-\infty}^\infty f(x) \overline{\psi \left(  \frac{x-b}{a}\right)  } dx,\,\, b \in \mathbb{R}, a>0,
\end{eqnarray}
provided the integral exists \cite{rspathak} \\
Now, from (\ref{1}) we get
\begin{eqnarray}
\label{2}
\left(  W_\psi f \right)  (a,b) \nonumber &=& \sqrt{a} \int_{-\infty}^\infty f(x) \overline{\psi \left(  x- \frac{b}{a} \right) } dx \\
&=& \sqrt{a} \bigg\langle f(ax), \psi\left( x- \frac{b}{a}\right) \bigg\rangle
\end{eqnarray}
 This paper is arranged in following manner. In section second, third , fourth and fifth we drive the asymptotic expansion in the distributional spaces $\mathscr{E}'(\mathbb{R}) $, $ \mathscr{P}'(\mathbb{R})$,$\mathcal{O}'_\gamma (\mathbb{R})$,$\mathcal{O}'_c(\mathbb{R})$ and $\mathcal{O}'_M(\mathbb{R})$ respectively, studied in \cite{kanwal}
\section{The moment asymptotic expansion of $\left(  W_\psi f \right)  (a,b) $ as $a\rightarrow \infty $ in the space $\mathscr{E}'(\mathbb{R})$ for given $b$ }
The space $\mathscr{E}(\mathbb{R})$ is the space of all smooth functions on $\mathbb{R}$ and it's dual space $\mathscr{E}'(\mathbb{R})$ , the space of distribution with compact support.   If  $\psi \in \mathscr{E}(\mathbb{R})$, then $ \psi \left(  x- \frac{b}{a} \right) \in \mathscr{E}(\mathbb{R})$. So consider the seminorms \\
\begin{itemize}
\item[Case 1] For $ b\geq 0 $
\begin{eqnarray}
\label{3}
\bigg\Vert\psi \left(  x- \frac{b}{a} \right) \bigg\Vert_{\alpha,M}=Max\bigg\lbrace \vert D^\alpha \psi \left(  x- \frac{b}{a} \right)\vert : \frac{b}{a}-M<x<b+M  \bigg\rbrace
\end{eqnarray}
\end{itemize}
\begin{itemize}
\item[Case 2] For $ b < 0  $
\begin{eqnarray}
\label{4}
\bigg\Vert\psi \left(  x- \frac{b}{a} \right) \bigg\Vert_{\alpha,M}=Max\bigg\lbrace \vert D^\alpha \psi \left(  x- \frac{b}{a} \right)\vert :  b-M<x< \frac{b}{a}+M  \bigg\rbrace
\end{eqnarray}
\end{itemize}
for $\alpha \in \mathbb{N} $ and $M>0 $, these seminorm generate the topology of $\mathscr{E}(\mathbb{R})$. If $q=0,1,2,3, ...$, we set
\begin{eqnarray}
\label{5}
X_q=\lbrace \psi \in \mathscr{E}(\mathbb{R}) : D^\alpha \psi \left( 0 \right)=0 \,\, for \,\, \alpha < q   \rbrace
\end{eqnarray}
\begin{Lemma}
\label{6}
Let $\psi \in X_q $, then for every $\alpha \in \mathbb{N} $ and $ M>0 $,
\begin{eqnarray}
\label{7}
\bigg\Vert \psi \left(\frac{x-b}{a} \right)\bigg\Vert_{\alpha,M}=O\left( \frac{1}{a^q}\right) \,\,\, as \,\,\, a \rightarrow \infty
\end{eqnarray}
\begin{pro}
\label{7}
For $ b\geq 0 $. For $\psi \in X_q $ we can find a constant $K$ such that
\begin{eqnarray}
\label{8}
\bigg\vert \psi \left(  x- \frac{b}{a} \right)\bigg\vert  \leq K \bigg\vert x- \frac{b}{a} \bigg\vert^q
\,\, ,\frac{b}{a}-1<x<b+1.
\end{eqnarray}
Therefore, if $a>M$ we obtain
\begin{eqnarray}
\label{9}
\bigg\Vert \psi \left( \frac{x-b}{a} \right)\bigg\Vert_{0,M}\nonumber &=& Max\bigg\lbrace \bigg\vert \psi \left( \frac{x-b}{a} \right)\bigg\vert : \frac{b}{a}-M<x<b+M  \bigg\rbrace \\ & \leq & O\left( \frac{M}{a^q}\right).
\end{eqnarray}
If $\alpha \leq q $ and $\psi \in X_q $ then $ D^\alpha \psi \in X_{q-\alpha}$ and thus
\begin{eqnarray*}
\bigg\Vert \psi \left(\frac{x-b}{a} \right)\bigg\Vert_{\alpha,M}  &=& \bigg\Vert a^\alpha D^\alpha \psi \left(  \frac{x-b}{a} \right)\bigg\Vert_{0,M} \\ &=& \frac{1}{a^\alpha} O \left( \frac{1}{a^{q-\alpha}}\right) \\ &=& O \left( \frac{1}{a^{q}}\right)
\end{eqnarray*}
Similarly by using (\ref{4}) we can prove that
\begin{eqnarray*}
\bigg\Vert \psi \left(\frac{x-b}{a} \right)\bigg\Vert_{\alpha,M}  = O \left( \frac{1}{a^{q}}\right)  \;\; for \;\; b<0.
\end{eqnarray*}
\end{pro}
\end{Lemma}
Now, by using Lemma \ref{6} we obtain the following theorem
\begin{thm}
\label{1.1}
Let wavelet $\psi \in \mathscr{E}(\mathbb{R}) , f \in \mathscr{E}'(\mathbb{R})$ and $\mu_\alpha = \left\langle f, x^\alpha \right\rangle $ be its moment sequence . Then for a fixed $b$ the moment asymptotic expansion of wavelet transform is
\begin{eqnarray}
\label{10}
\sqrt{a}\bigg\langle f(ax), \psi\left(  x- \frac{b}{a}\right) \bigg\rangle= \sum_{\alpha =0}^{N} \frac{\mu_\alpha D^\alpha\psi(-b/a)}{\alpha ! \,\, a^{\alpha+ 1/2}} + O\left( \frac{1}{a^{N+1/2}}\right) \;\; as \;\; a \rightarrow \infty .
\end{eqnarray}
\end{thm}
\begin{pro}
Let $P_N(x,b/a)=\sum_{\alpha =0}^{N} \frac{D^\alpha\psi(-b/a)}{\alpha !} x^\alpha $ be the polynomial of order $N$ of the function $\psi\left( x- \frac{b}{a}\right) $. Then we have
\begin{eqnarray*}
\label{11}
\bigg\langle f(ax), \psi\left(  x- \frac{b}{a}\right) \bigg\rangle & =&  \bigg\langle f(ax), P_N(x,b/a) \bigg\rangle + \bigg\langle f(ax), \psi\left( x-\frac{b}{a} \right)-P_N(x,b/a) \bigg\rangle \\ &=& \sum_{\alpha=0}^{N} \frac{\mu_\alpha D^{\alpha} \psi (-b/a)}{\alpha ! a^{\alpha+1}}+ R_N (a)
\end{eqnarray*}
where the remainder $R_N (a)$ is given as $R_N (a)= \bigg\langle f(ax), \psi\left( x-\frac{b}{a} \right)-P_N(x,b/a) \bigg\rangle$.\\
Since $\psi\left( x-\frac{b}{a} \right)-P_N(x,b/a)  \ in X_{N+1}$ we obtain
\begin{eqnarray*}
|R_N(a)| & = & \left| \bigg\langle f(ax), \psi\left( x-\frac{b}{a} \right)-P_N(x,b/a) \right| \\ &=& \frac{L}{a} \sum_{\alpha=0}^{q} \| \psi_{N} \left( \frac{x-b}{a}\right) \|_{\alpha,M} \\ &=& O\left(\frac{1}{a^{N+1}}\right)
\end{eqnarray*}
where the existence of $ L, q $ and $M$ is guaranteed by the continuity of $f$. Hence we get the required asymptotic expansion \ref{10}.
\end{pro}
\begin{exam}
In this example we choose $\psi$ to be Mexican-Hat wavelet and derive the asymptotic expansion of Mexican-Hat wavelet transform by using Theorem\ref{1.1}. The Mexican-Hat wavelet is given by \cite{rspathak}
\begin{eqnarray}
\psi(x)= (1-x^2) e^{- \frac{x^2}{2}} \in \mathscr{E}(\mathbb{R})
\end{eqnarray}
Let
\begin{eqnarray*}
P_2(x,b/a)=\frac{e^{-\frac{b^2}{2a^2}}}{a^2} \left( (a^2-b^2)+ \frac{b(3a^2-b^2)}{a}x+\frac{(6a^2b^2-3a^4-b^4)}{2a^2}x^2 \right).
\end{eqnarray*}
Now, using Theorem\ref{1.1} we get the asymptotic expansion of Mexican-Hat wavelet transform
\begin{eqnarray*}
\sqrt{a}\bigg\langle f(ax), \psi\left(  x- \frac{b}{a}\right) \bigg\rangle &=& \frac{e^{-\frac{b^2}{2a^2}}}{a^2} \bigg( \frac{(a^2-b^2)}{\sqrt{a}}  \mu_0+ \frac{b(3a^2-b^2)}{a^{3/2}}\mu_1  \\ &+& \frac{(6a^2b^2-3a^4-b^4)}{2a^{5/2}}\mu_2 \bigg) + O \left(\frac{1}{a^{9/2}}\right) \,\,\, as\,\,\, a \rightarrow \infty
\end{eqnarray*}
where $ \mu_{i}= \langle f,x^i\rangle , i=0,1,2$.
\end{exam}
\section{The moment asymptotic expansion of $\left(  W_\psi f \right)  (a,b) $ for large and small values  of $a$ in the space  $\mathscr{P}'(\mathbb{R})$ for a given $b$ }
Case 1. Let $ \psi \in \mathscr{P}(\mathbb{R})$. \\ We now consider the moment asymptotic expansion in the space $\mathscr{P}'(\mathbb{R})$ of distributions of "less than exponential growth". The space $\mathscr{P}(\mathbb{R})$ consist of those smooth functions $\phi(x)$ that satisfy
\begin{eqnarray*}
lim_{x \rightarrow \infty} e^{-\gamma |x|} D^{\beta} \phi(x)= 0 \,\,\, for \,\,\, \gamma >0 \,\, and \,\, each \,\, \beta \in \mathbb{N},
\end{eqnarray*}
with seminorms
\begin{eqnarray*}
\| \phi(x) \|_{\gamma, \beta} = sup \,\, \left\{ | \,\,\, e^{-\gamma |x|} \,\, D^{\beta} \phi(x) \,\, | : x \in  \mathbb{R} \right\}.
\end{eqnarray*}
Let wavelet $\psi(x) \in \mathscr{P}(\mathbb{R}) $. Then
\begin{eqnarray*}
\| \psi(x) \|_{\,\, \gamma,\,\, \beta, \,\, \frac{b}{a}} &=& sup \left\{ \left| e^{-\gamma |x|} D^{\beta} \psi \left( x-\frac{b}{a}\right) \right| : x \in  \mathbb{R} \right\} \\ &=& sup \left\{ \left| e^{-\gamma |x- \frac{b}{a}|} D^{\beta} \psi \left( x-\frac{b}{a}\right)  \frac{e^{-\gamma |x|}}{e^{-\gamma |x- \frac{b}{a}|}} \right| : x \in  \mathbb{R} \right\} \\ &=& \bigg\| \psi \left( x-\frac{b}{a}\right) \bigg\|_{\,\, \gamma,\,\, \beta} \,\, A \left( x,b/a\right),
\end{eqnarray*}
where
\begin{eqnarray*}
A \left( x,b/a\right) &=& sup \left\{ \left| \frac{e^{-\gamma |x|}}{e^{-\gamma |x- \frac{b}{a}|}} \right| : x \in  \mathbb{R} \right\}
\\ & \leq & e^{\gamma \left| \frac{b}{a} \right|} < \infty,
\end{eqnarray*}
for a given $ \gamma > 0$ and $b \in \mathbb{R}$. \\
So $\| \psi \left( x \right) \|_{\,\, \gamma,\,\, \beta, \,\,\frac{b}{a}}$ is also seminorm on $ \mathscr{P}(\mathbb{R}) $ for $\gamma >0,  \,\, \beta \in \mathbb{N}$ and for a given $b \in \mathbb{R} $. Therefore these seminorm generate the topology of the space $\mathscr{P}(\mathbb{R}) $. If
\begin{eqnarray*}
X_q = \lbrace \psi \in \mathscr{P}(\mathbb{R}) : D^\alpha \psi(0)=0, \,\, for \,\,\, \alpha < q \rbrace.
\end{eqnarray*}
Therefore for any $\gamma >0 $ we can find a constant $C$ such that
\begin{eqnarray*}
  \left| \psi \left( x-\frac{b}{a}\right) \right| \leq C \left| x- \frac{b}{a} \right|^{q}  e^{\frac{\gamma |x|}{2}} e^{\gamma \left| \frac{b}{a} \right|}
\end{eqnarray*}
if $a>1$
\begin{eqnarray*}
  e^{-\gamma |x|} \left| \psi \left( \frac{x-b}{a}\right) \right| \leq C \left| \frac{x-b}{a} \right|^{q}  e^{-\frac{\gamma |x|}{2}} e^{\gamma \left| \frac{b}{a} \right|} \leq  \frac{C_1}{a^q}
\end{eqnarray*}
and thus
\begin{eqnarray}
\label{3.1}
  \bigg\Vert \psi \left( \frac{x}{a} \right) \bigg\Vert_{\gamma,\,\, 0,\,\, \frac{b}{a}}= O \left( \frac{1}{a^q}\right) \,\, as \,\, a \rightarrow \infty,\,\, \psi \in X_q.
\end{eqnarray}
Hence using above equation we get
\begin{eqnarray}
\label{3.2}
  \bigg\Vert \psi \left( \frac{x}{a} \right) \bigg\Vert_{\gamma,\,\, \beta, \,\,  \frac{b}{a}}= O \left( \frac{1}{a^q}\right) \,\, as \,\, a \rightarrow \infty.
\end{eqnarray}
Using (\ref{3.2}) we obtain the following theorem
\begin{thm}
\label{33.3}
Let $ \psi \in \mathscr{P}(\mathbb{R}) , f \in \mathscr{P}'(\mathbb{R})$  and $\mu_{\alpha}=\langle f, x^\alpha \rangle $ be its moment sequence. Then for a fixed $b$ the asymptotic expansion of wavelet transform is
\begin{eqnarray}
\label{3.3}
\sqrt{a}\bigg\langle f(ax), \psi\left(  x- \frac{b}{a}\right) \bigg\rangle  \sim  \sum_{\alpha=0}^{\infty} \frac{\mu_\alpha D^\alpha \psi(-b/a)}{\alpha ! \,\, a^{\alpha +1/2}} \,\,\, as \,\,\, a \rightarrow \infty.
\end{eqnarray}
\end{thm}
\begin{pro} Similarly as Theorem \ref{1.1}
\end{pro}
\begin{exam}
Let $\psi(x)= (1-x^2) e^{- \frac{x^2}{2}} \in \mathscr{P}(\mathbb{R})$ is Mexican-Hat wavelet and $ f(x) \in \mathscr{P}'(\mathbb{R}) $. Therefore by Theorem \ref{3.3} moment asymptotic expansion of continuous Mexican-Hat wavelet transform for large $a$ in  $\mathscr{P}'(\mathbb{R})$ is given by
\begin{eqnarray*}
\sqrt{a}\bigg\langle f(ax), \psi\left(  x- \frac{b}{a}\right) \bigg\rangle \sim  \sum_{\alpha=0}^{\infty} \frac{\mu_\alpha D^\alpha [(1-x^2) e^{-\frac{x^2}{2}}]_{x=- \frac{b}{a}}}{\alpha ! \,\, a^{\alpha +1/2}} \,\,\, as \,\,\, a \rightarrow \infty.
\end{eqnarray*}
\end{exam}
Case 2. In this case we consider  wavelet $\psi(x) \in \mathscr{P}'(\mathbb{R}) $ and $f(x) \in \mathscr{P}(\mathbb{R}) $. Then the wavelet transform (\ref{1}) can we rewrite as
 \begin{eqnarray*}
\label{2}
\left(  W_\psi f \right)  (a,b) \nonumber &=& \frac{1}{\sqrt{a}}  \bigg\langle \psi\left(  \frac{x}{a}\right),  f(x+b)  \bigg\rangle
\end{eqnarray*}
Similarly as Theorem \ref{33.3} we can also obtain the following theorem
\begin{thm}
\label{3333}
Let $ \psi \in \mathscr{P}'(\mathbb{R}) , f \in \mathscr{P}(\mathbb{R})$  and $\mu_{\alpha}=\langle \psi, x^\alpha \rangle $ be its moment sequence. Then for a fixed $b$ the asymptotic expansion of wavelet transform is
\begin{eqnarray}
\label{3.4}
\frac{1}{\sqrt{a}}  \bigg\langle \psi\left(  \frac{x}{a}\right),  f(x+b)  \bigg\rangle \sim  \sum_{\alpha=0}^{\infty} \frac{\mu_\alpha D^\alpha f(b)a^{\alpha +1/2}}{\alpha ! } \,\,\, as \,\,\, a \rightarrow 0.
\end{eqnarray}
\end{thm}
\begin{exam}
In this example again we consider the Mexican-Hat wavelet which is less then exponential growth , so by applying Theorem \ref{3333} and using formula [30,pp.320, 1], we get the asymptotic expansion of wavelet transform for small values of $a$
\begin{eqnarray*}
\frac{1}{\sqrt{a}}  \bigg\langle \psi\left(  \frac{x}{a}\right),  f(x+b)  \bigg\rangle  \sim  \sum_{\alpha=0}^{\infty} -2^{\frac{1}{2} (2\alpha-1)} \Gamma\left( \frac{2\alpha + 1}{2}\right) \frac{ D^{2\alpha} f(b)a^{2\alpha +1/2}}{(2\alpha) ! } \,\,\, as \,\,\, a \rightarrow 0.
\end{eqnarray*}
\end{exam}
\section{The moment asymptotic expansion of $\left(  W_\psi f \right)  (a,b) $ as $a\rightarrow \infty $ in the space $\mathcal{O}'_\gamma (\mathbb{R})$ for given $b$ }
A test function $\psi$ belongs to $\mathcal{O}_\gamma (\mathbb{R})$, if it is smooth and $D^\alpha \psi (x) = O (|x|^\gamma) $ as $ x  \rightarrow \infty $  for every $\alpha \in \mathbb{N} $ and $ \gamma \in \mathbb{R} $. The family of seminorms
 \begin{eqnarray*}
\Vert \psi(x) \Vert_{\alpha, \gamma} = sup \lbrace \rho_\gamma(|x|) \vert D^\alpha \psi(x) \vert : x \in \mathbb{R} \rbrace
\end{eqnarray*}
where
\begin{eqnarray}
\label{4.1}
 \rho_\gamma(|x|) =  \left\{\begin{array}{cc}
                 1 ,  & 0 \leq |x| \leq 1 \\
                                    |x|^{-\gamma}, & \,\,\,  |x|>1
                \end{array}\right\},
\end{eqnarray}
generates a topology for $\mathcal{O}_\gamma (\mathbb{R})$. Now with the help of the translation version of $\psi(x)$, we can define the seminorms on $\mathcal{O}_\gamma (\mathbb{R})$ as
\begin{eqnarray*}
\Vert \psi(x) \Vert_{\alpha, \gamma,  b/a} &=& sup \left\lbrace  \rho_\gamma(|x|) \bigg\vert D^\alpha \psi\bigg(x-\frac{b}{a}\bigg) \bigg\vert : x \in \mathbb{R} \right\rbrace  \\ &=& sup \left\lbrace  \rho_\gamma\bigg(\bigg| x-\frac{b}{a} \bigg|\bigg) \bigg\vert D^\alpha \psi\bigg(x-\frac{b}{a}\bigg) \bigg\vert : x \in \mathbb{R} \right\rbrace \nabla (x,b/a) \\ &=& \bigg\Vert \psi\bigg(x-\frac{b}{a}\bigg) \bigg\Vert_{\alpha, \gamma}\nabla (x,b/a)
\end{eqnarray*}
where $ \nabla (x,b/a) = sup \bigg\lbrace \frac{\rho_\gamma(|x|) }{\rho_\gamma(| x-\frac{b}{a}|)}: x \in \mathbb{R} \bigg\rbrace $. \\
 So, for $\gamma >0 $.
 \begin{eqnarray*}
\nabla (x,b/a) \leq \left\{\begin{array}{cccc}
                 1,  & for \,\, 0 \leq |x| \leq 1 \,\,\, and \,\, 0 \leq |x-b/a| \leq 1  \\
                                     \bigg( 1+\frac{|b/a|}{1-|b/a|}\bigg)^\gamma, & for \,\, |x|>1 \,\,\, and \,\, |x-b/a|>1
                                     \\
                                      (1+|\frac{b}{a}|)^\gamma ,  & for \,\, 0\leq |x| \leq 1\,\,\, and \,\, |x-b/a| >1 \\
                                    1, & for \,\, |x|>1 \,\,\, and \,\, |x-b/a| \leq 1
                \end{array}\right\}.
\end{eqnarray*}
Similarly for $\gamma < 0 $. we have
\begin{eqnarray*}
\nabla (x,b/a) \leq \left\{\begin{array}{cc}
                 1,  & for \,\, 0 \leq |x| \leq 1 \,\,\, and \,\, 0 \leq |x-b/a| \leq 1  \\
                                     \bigg( 1+|b/a|\bigg)^{-\gamma}, & otherwise

                \end{array}\right\}.
\end{eqnarray*}
Thus $ \,sup \bigg\lbrace \frac{\rho_\gamma(|x|) }{\rho_\gamma(| x-\frac{b}{a}|)}: x \in \mathbb{R} \bigg\rbrace  \leq \bigg( 1+|b/a|\bigg)^{|\gamma|},  =K < \infty  , \,\,\  \forall \,\,\,  \gamma  \,\, \in \mathbb{R} $\\
Therefore $ \|\psi(x)\|_{\alpha,\,\ \gamma, \,\, b/a}$ are also seminorm on  $\mathcal{O}_\gamma (\mathbb{R})$.These seminorm generate the topology of the space $\psi(x) \in \mathcal{O}_\gamma (\mathbb{R}) $.If
\begin{eqnarray*}
X_q = \lbrace \psi \in \mathcal{O}_\gamma (\mathbb{R}) : D^\alpha \psi(0)=0, \,\, for \,\,\, \alpha < q \rbrace.
\end{eqnarray*}
So for any $\gamma $ we can find a constant $C$ such that
\begin{eqnarray*}
 \rho(|x|) \left| \psi \left( x-\frac{b}{a}\right) \right| \leq C \rho(|x|) \left| x- \frac{b}{a} \right|^{q}  \nabla (x,b/a).
\end{eqnarray*}
If $a>1$
\begin{eqnarray*}
 \rho(|x|) \left| \psi \left( x-\frac{b}{a}\right) \right| \leq  \frac{M}{a^q}
\end{eqnarray*}
Hence using above equation we get
\begin{eqnarray}
\label{4.1}
  \bigg\Vert \psi \left( \frac{x}{a} \right) \bigg\Vert_{\alpha,\,\ \gamma\,\, ,  b/a}= O \left( \frac{1}{a^q}\right) \,\, as \,\, a \rightarrow \infty.
\end{eqnarray}
Similarly as Theorem \ref{33.3} we can obtain the following theorem
\begin{thm}
Let $ \psi \in \mathcal{O}_\gamma (\mathbb{R}) , f \in \mathcal{O}'_\gamma (\mathbb{R}), \,\, N=[[\gamma]]-1$  and $\mu_{\alpha}=\langle f, x^\alpha \rangle $ be its moment sequence. Then for a fixed $b \in \mathbb{R} $ the asymptotic expansion of wavelet transform is
\begin{eqnarray}
\label{4.2}
  \sqrt{a}\bigg\langle f(ax), \psi\left(  x- \frac{b}{a}\right) \bigg\rangle= \sum_{\alpha=0}^{N} \frac{\mu_\alpha D^\alpha \psi(-b/a)}{\alpha ! a^{\alpha +1/2}} + O \bigg( \frac{1}{a^{N+1/2}} \bigg)\,\,\, as \,\,\, a \rightarrow \infty.
\end{eqnarray}
\end{thm}
Since $ \mathcal{O}'_c (\mathbb{R}) =\bigcap  \mathcal{O}'_\gamma (\mathbb{R}) $, we obtain the asymptotic expansion of wavelet transform in the space $ \mathcal{O}'_c (\mathbb{R}) $
\begin{thm}
Let $ \psi \in \mathcal{O}_c (\mathbb{R}) , f \in \mathcal{O}'_c (\mathbb{R}) $  and $\mu_{\alpha}=\langle f, x^\alpha \rangle $ be its moment sequence. Then for a fixed $b \in \mathbb{R} $ the asymptotic expansion of wavelet transform is
\begin{eqnarray}
\label{4.3}
  \sqrt{a}\bigg\langle f(ax), \psi\left(  x- \frac{b}{a}\right) \bigg\rangle \sim \sum_{\alpha=0}^{\infty} \frac{\mu_\alpha D^\alpha \psi(-b/a)}{\alpha ! a^{\alpha +1/2}} + O \bigg( \frac{1}{a^{N+1/2}} \bigg)\,\,\, as \,\,\, a \rightarrow \infty.
\end{eqnarray}
\end{thm}
\section{The moment asymptotic expansion of $\left(  W_\psi f \right)  (a,b) $ as $a\rightarrow \infty $ in the space $\mathcal{O}'_M (\mathbb{R})$ for given $b$ }
The space $\mathcal{O}_M (\mathbb{R})$ consist of all $ c^\infty $-function whose derivatives are bounded by polynomials ( of probably different degrees). Let $ \psi \in \mathcal{O}_M (\mathbb{R}) $ then its translation version is also in $\mathcal{O}_M (\mathbb{R})$. Then by using Theorem 9 \cite{kanwal} we can also derive the asymptotic expansion of wavelet transform in $\mathcal{O}'_M (\mathbb{R})$
\begin{thm}
Let $ \psi \in \mathcal{O}_M (\mathbb{R}) , f \in \mathcal{O}'_M (\mathbb{R}) $  and $\mu_{\alpha}=\langle f, x^\alpha \rangle $ be its moment sequence. Then for a fixed $b \in \mathbb{R} $ the asymptotic expansion of wavelet transform is
\begin{eqnarray}
\label{5.1}
 \langle f(x), \psi_{a,b}(x)\rangle \sim \sum_{\alpha=0}^{\infty} \frac{\mu_\alpha (f) D^{\alpha}\psi(-\frac{b}{a})}{\alpha ! \,\,\,  a^{\alpha +1/2}} \,\,\, as \,\,\, a \rightarrow \infty.
\end{eqnarray}
\end{thm}
\begin{proof}
By using (1.7.1) \cite{rspathak}  we can be write the wavelet transform
\begin{eqnarray}
\label{5.2}
\sqrt{a}\bigg\langle f(ax), \psi\left(  x- \frac{b}{a}\right) \bigg\rangle =  \sqrt{a} \langle e^{ib\omega}\hat{f}(\omega), \hat{\psi}(a\omega)\rangle
\end{eqnarray}
 where $\psi(x) \in \mathcal{O}_M (\mathbb{R}) $ and $ f(x) \in \mathcal{O}'_M (\mathbb{R})$ then its Fourier transforms  $\hat{\psi}(\omega) \in \mathcal{O}'_c (\mathbb{R}) $ and $ \hat{f}(\omega) \in \mathcal{O}_c (\mathbb{R})$ respectively.\\
Now by using Theorem 4.2 we get
\begin{eqnarray}
\label{5.3}
\langle f(x), \psi_{a,b}(x)\rangle \sim \sum_{\alpha=0}^{\infty} \frac{\mu_\alpha ( e^{-i\frac{b}{a}\omega}\hat{\psi}(\omega)) D^\alpha ( \hat{f}(0)}{\alpha ! \,\,\,  a^{\alpha +1/2}} \,\,\, as \,\,\, a \rightarrow \infty.
\end{eqnarray}
But by the properties of Fourier transform we have
\begin{eqnarray*}
\mu_\alpha (e^{-i{b}{a}\omega} \hat{\psi}(\omega)) =  \langle e^{-i\frac{b}{a}\omega} \hat{\psi}(\omega), \omega^\alpha \rangle = i^{-\alpha} D^{\alpha}\psi\bigg(-\frac{b}{a}\bigg), \,\,\,
D^\alpha ( \hat{f}(\omega))_{\omega=0} = i^{\alpha}  \mu_\alpha (f(x))
\end{eqnarray*}
and hence
\begin{eqnarray}
\label{5.4}
\langle f(x), \psi_{a,b}(x)\rangle \sim \sum_{\alpha=0}^{\infty} \frac{\mu_\alpha (f) D^{\alpha}\psi(-\frac{b}{a})}{\alpha ! \,\,\,  a^{\alpha +1/2}} \,\,\, as \,\,\, a \rightarrow \infty.
\end{eqnarray}
\end{proof}
\section*{Acknowledgment}
The work of second author is supported by U.G.C. start-up grant.
\end{document}